\documentclass[11pt,reqno]{amsart}

\usepackage[T1]{fontenc}
\usepackage{lmodern}
\usepackage{microtype}
\usepackage{amsmath,amssymb,amsthm,mathtools}
\usepackage{booktabs,array}
\usepackage{placeins}
\usepackage{graphicx}
\usepackage{enumitem}
\usepackage{hyperref}
\usepackage{aliascnt}
\usepackage[nameinlink,capitalize,noabbrev]{cleveref}
\usepackage[a4paper,margin=32mm]{geometry}

\allowdisplaybreaks[3]
\numberwithin{equation}{section}
\hypersetup{
  colorlinks=true,
  linkcolor=blue,
  citecolor=blue,
  urlcolor=blue,
  pdftitle={The 196560 auxiliary-function conjecture for the Leech lattice},
  pdfauthor={Yutong Zhang and Yaoran Yang},
  pdfsubject={Fourier interpolation and the Leech lattice},
  pdfkeywords={Fourier interpolation, Leech lattice, radial Schwartz function, modular forms}
}

\newtheorem{theorem}{Theorem}[section]
\newaliascnt{proposition}{theorem}
\newtheorem{proposition}[proposition]{Proposition}
\aliascntresetthe{proposition}
\newaliascnt{lemma}{theorem}
\newtheorem{lemma}[lemma]{Lemma}
\aliascntresetthe{lemma}
\newaliascnt{corollary}{theorem}
\newtheorem{corollary}[corollary]{Corollary}
\aliascntresetthe{corollary}
\newaliascnt{remark}{theorem}

\aliascntresetthe{remark}
\newaliascnt{definition}{theorem}

\aliascntresetthe{definition}

\newcommand{\R}{\mathbb R}
\newcommand{\C}{\mathbb C}
\newcommand{\Z}{\mathbb Z}

\newcommand{\Sch}{\mathcal S}
\newcommand{\rad}{\mathrm{rad}}
\newcommand{\Leech}{\Lambda_{24}}
\newcommand{\dd}{\,\mathrm d}
\newcommand{\e}{\mathrm e}
\newcommand{\ii}{\mathrm i}
\newcommand{\diag}{\operatorname{diag}}
\newcommand{\Oh}{O}

\title[The 196560 auxiliary-function conjecture]{The 196560 auxiliary-function conjecture for the Leech lattice}

\author{Yutong Zhang}
\address{School of Mathematics, Sichuan University,
24 First Loop Road South Section I,
Chengdu 610064, Sichuan, China}
\email{yutongzhang@stu.scu.edu.cn}
\thanks{Corresponding author: Yutong Zhang}

\author{Yaoran Yang}
\address{School of Mathematics, Sichuan University,
24 First Loop Road South Section I,
Chengdu 610064, Sichuan, China}
\email{yangyaoran@stu.scu.edu.cn}


\begin{document}

\begin{abstract}
Cohn and Kumar conjectured in 2009 that there is a radial Schwartz function
$g\colon\R^{24}\to\R$ satisfying $g(r)\leq0$ for $r\geq\sqrt6$,
$\widehat g(r)\geq0$ for $r\geq0$, $g(2)>0$, and
$(\widehat g(0)-g(0))/g(2)=196560$.  We construct such functions from the
radial Fourier interpolation basis in dimension $24$.  If $a_2,b_2$ denote
the basis functions dual to value and radial-derivative interpolation at radius
$2$, then $g_C=a_2-Cb_2$ has exactly the nodal data needed for Poisson
summation over the Leech lattice.  The sphere-packing magic function identifies
$b_2$ and supplies its strict signs.  We prove that the removable quotients
$a_2/b_2$ and $\widehat a_2/\widehat b_2$ are bounded on the required
half-lines.  The noncompact step follows from exact coefficient extraction in
the interpolation kernel and an $S$-cusp expansion.  Both quotients tend to
$(43+240\log2)/15$; for all sufficiently large radii they lie on opposite
sides of this limit, with an explicit first exponential correction.
Consequently the admissible parameters in this affine family form a nonempty
closed ray, and every member proves the conjectured identity.  The same
interpolation basis recovers every nontrivial Leech-shell coefficient.
\end{abstract}

\maketitle

\section{Introduction and main results}

We use the Fourier transform convention
$\widehat h(y)=\int_{\R^{24}}h(x)\e^{-2\pi\ii\langle x,y\rangle}\dd x$.
For a radial function we write $h(r)$ for its value on $|x|=r$ and $h'(r)$
for its radial derivative.

Cohn and Kumar explicitly posed the following problem in their proof of the
lattice optimality of the Leech lattice \cite{CK2009}.  Find a radial Schwartz
function $g\colon\R^{24}\to\R$ such that
\begin{align}
 g(r)&\leq0 &&(r\geq\sqrt6), \label{eq:CK-direct-sign}\\
 \widehat g(r)&\geq0 &&(r\geq0), \label{eq:CK-fourier-sign}\\
 g(2)&>0, \label{eq:CK-positive-value}
\end{align}
and
\begin{equation}\label{eq:CK-ratio}
  \frac{\widehat g(0)-g(0)}{g(2)}=196560.
\end{equation}
The number in \eqref{eq:CK-ratio} is forced by the $196560$ shortest vectors
of the Leech lattice and is the largest value compatible with the Poisson
lower-bound argument of \cite{CK2009}.  The surrounding linear-programming
method originated in \cite{CE2003}.

The radial interpolation theorem of Cohn, Kumar, Miller, Radchenko, and
Viazovska \cite{CKMRV2022} supplies Schwartz functions
$a_n,b_n,\widetilde a_n,\widetilde b_n\in\Sch_{\rad}(\R^{24})$ for $n\geq2$,
which are dual to the values and first radial derivatives of a function and its
Fourier transform at the radii $\sqrt{2n}$.  Our construction uses only the
first value--derivative pair:
\begin{equation}\label{eq:gC-definition}
  g_C:=a_2-Cb_2.
\end{equation}
The sign anchor is the optimal sphere-packing auxiliary function from
\cite{CKMRV2017}; it is a nonzero scalar multiple of $b_2$ by interpolation
uniqueness.  The remaining issue is the boundedness of two quotients on
noncompact intervals.

Define the quotients initially away from their common interpolation zeros by
\begin{equation}\label{eq:R-definitions}
 R(r):=\frac{a_2(r)}{b_2(r)}\quad(r\geq\sqrt6),
 \qquad
 \widehat R(r):=\frac{\widehat a_2(r)}{\widehat b_2(r)}\quad(r\geq0).
\end{equation}
We prove that every apparent singularity in \eqref{eq:R-definitions} is
removable, both functions extend smoothly, and
\begin{equation}\label{eq:rho-definition}
 \lim_{r\to\infty}R(r)
 =\lim_{r\to\infty}\widehat R(r)
 =\rho,
 \qquad
 \rho:=\frac{43+240\log2}{15}.
\end{equation}
Thus the following finite threshold is well defined:
\begin{equation}\label{eq:Cstar-definition}
 C_*:=\max\left\{
   \sup_{r\geq\sqrt6}R(r),
   \sup_{r\geq0}\widehat R(r)
 \right\}.
\end{equation}

\begin{theorem}[The $196560$ auxiliary-function conjecture]\label{thm:main}
For every real $C\geq C_*$, the radial Schwartz function $g_C=a_2-Cb_2$
satisfies \eqref{eq:CK-direct-sign}--\eqref{eq:CK-ratio}; in fact $g_C(2)=1$.
Moreover, within the affine line $\{a_2-Cb_2:C\in\R\}$, the set of admissible
parameters is exactly the closed ray $[C_*,\infty)$.
\end{theorem}

The theorem gives more than existence: it identifies the complete feasible set
in the natural affine family determined by the interpolation data.  No numerical
approximation to $C_*$ is required.  The refined asymptotic analysis below
shows that
\begin{equation}\label{eq:Cstar-strict-lower}
 C_*>\rho>0.
\end{equation}

A sentence following the original conjecture suggested that the same method
should recover every coefficient of the Leech theta series.  Fourier
interpolation makes this exact.

\begin{corollary}[Recovery of every Leech shell]\label{cor:theta-shells}
For $m\geq2$, let $N_m:=\#\{x\in\Leech:|x|^2=2m\}$.  Then
\begin{equation}\label{eq:shell-interpolation}
 N_m=\widehat a_m(0)-a_m(0).
\end{equation}
Equivalently,
\begin{equation}\label{eq:shell-ramanujan}
 N_m=\frac{65520}{691}\bigl(\sigma_{11}(m)-\tau(m)\bigr),
\end{equation}
where $\tau(m)$ is Ramanujan's tau function.
\end{corollary}

The proof has three nonformal inputs.  First, Fourier interpolation gives the
basis and exact nodal data.  Second, the dimension-$24$ sphere-packing
function gives the signs of $b_2$ and $\widehat b_2$.  Third, an exact
calculation in the interpolation kernel gives the $S$-cusp asymptotics needed
to control infinity.  The coefficient extraction and every cusp cancellation
used in the third step are displayed and proved below as finite identities in
formal power series; no computer-assisted verification is needed for this
third step.

\section{Fourier interpolation and the sign anchor}

\subsection{The dimension-\texorpdfstring{$24$}{24} interpolation basis}

Let $\Sch_{\rad}(\R^{24})$ denote the radial Schwartz space.  The interpolation
theorem of \cite{CKMRV2022} states that every $h\in\Sch_{\rad}(\R^{24})$ has
the absolutely convergent expansion
\begin{align}
 h(x)
  ={}&\sum_{n=2}^{\infty}h(\sqrt{2n})a_n(x)
     +\sum_{n=2}^{\infty}h'(\sqrt{2n})b_n(x) \notag\\
    &+\sum_{n=2}^{\infty}\widehat h(\sqrt{2n})\widetilde a_n(x)
     +\sum_{n=2}^{\infty}\widehat h'(\sqrt{2n})\widetilde b_n(x).
 \label{eq:interpolation-formula}
\end{align}
The basis functions are characterized by the Kronecker data, for $m,n\geq2$,
\begin{align}
 a_n(\sqrt{2m})&=\delta_{m,n},
 &a_n'(\sqrt{2m})&=0,
 &\widehat a_n(\sqrt{2m})&=0,
 &\widehat a_n'(\sqrt{2m})&=0, \label{eq:a-data}\\
 b_n(\sqrt{2m})&=0,
 &b_n'(\sqrt{2m})&=\delta_{m,n},
 &\widehat b_n(\sqrt{2m})&=0,
 &\widehat b_n'(\sqrt{2m})&=0, \label{eq:b-data}
\end{align}
and $\widetilde a_n=\widehat a_n$ and
$\widetilde b_n=\widehat b_n$.

\begin{lemma}\label{lem:real-basis}
The functions $a_n,b_n,\widehat a_n,\widehat b_n$ are real-valued on
$\R^{24}$.
\end{lemma}

\begin{proof}
For every Schwartz function $u$ one has
$\widehat{\overline u}(y)=\overline{\widehat u(-y)}$.  Since radial functions
are even, the conjugates $\overline{a_n}$ and
$\overline{b_n}$ have exactly the same value, derivative, Fourier-value, and
Fourier-derivative interpolation data as $a_n$ and $b_n$, respectively.
Uniqueness in \eqref{eq:interpolation-formula} therefore gives
$\overline{a_n}=a_n$ and $\overline{b_n}=b_n$.  Finally, the Fourier transform
of a real even function is real and even, because
$\overline{\widehat u(y)}=\widehat u(-y)=\widehat u(y)$.
\end{proof}

\subsection{Identification and signs of \texorpdfstring{$b_2$}{b2}}

Let $f_{\mathrm{sp}}$ be the optimal auxiliary function for sphere packing in
$\R^{24}$ constructed in \cite{CKMRV2017}.  The construction extends
Viazovska's modular-form method for $E_8$ \cite{Viazovska2017}.  Its relevant
interpolation data are
\begin{align}
 f_{\mathrm{sp}}(\sqrt{2n})&=0 &&(n\geq2),\label{eq:fsp-values}\\
 f_{\mathrm{sp}}'(2)&=-\frac1{16380},
 &f_{\mathrm{sp}}'(\sqrt{2n})&=0 &&(n\geq3),\label{eq:fsp-derivatives}\\
 \widehat f_{\mathrm{sp}}(\sqrt{2n})&=0,
 &\widehat f_{\mathrm{sp}}'(\sqrt{2n})&=0 &&(n\geq2).
 \label{eq:fsp-fourier-data}
\end{align}
These conditions also appeared in the numerical and structural study of
optimal functions in \cite{CM2016}.

\begin{proposition}[Sign anchor]\label{prop:sign-anchor}
One has
\begin{equation}\label{eq:b2-fsp}
 b_2=-16380 f_{\mathrm{sp}}.
\end{equation}
Furthermore,
\begin{align}
 b_2(r)&>0
 &&\text{for }r>2\text{ and }r^2\notin2\Z, \label{eq:b2-positive}\\
 \widehat b_2(r)&<0
 &&\text{for }r\geq0\text{ and }r^2\notin2\Z_{\geq2}. \label{eq:hatb2-negative}
\end{align}
At $r=\sqrt{2n}$, $n\geq3$, the zeros of $b_2$ are exactly double; at
$r=\sqrt{2n}$, $n\geq2$, the zeros of $\widehat b_2$ are exactly double.
\end{proposition}

\begin{proof}
Equations \eqref{eq:fsp-values}--\eqref{eq:fsp-fourier-data} and the
interpolation formula \eqref{eq:interpolation-formula} give
$f_{\mathrm{sp}}=-b_2/16380$, which proves \eqref{eq:b2-fsp}.

For $r>2$, the direct-side formula in \cite{CKMRV2017} is
\begin{equation}\label{eq:fsp-direct-integral}
 f_{\mathrm{sp}}(r)
 =\sin^2\!\left(\frac{\pi r^2}{2}\right)
   \int_0^\infty A(t)\e^{-\pi r^2t}\dd t,
\end{equation}
where $A$ is continuous on $(0,\infty)$, satisfies $A(t)\leq0$, and is not
identically zero.  Hence $A(t_0)<0$ for some $t_0>0$, and continuity gives an
open interval on which $A<0$.  Since the Laplace kernel is strictly positive,
$\int_0^\infty A(t)\e^{-\pi r^2t}\dd t<0$ for every $r>2$.  Thus
$f_{\mathrm{sp}}(r)<0$ away from the zeros of the sine-square factor.  Those
zeros are exactly double, and the integral never vanishes; therefore the zeros of
$f_{\mathrm{sp}}$ at $r=\sqrt{2n}$, $n\geq3$, are exactly double.
Multiplication by $-16380$ proves \eqref{eq:b2-positive} and the corresponding
multiplicity assertion.

For $r>\sqrt2$, the Fourier-side formula is
\begin{equation}\label{eq:fsp-fourier-integral}
 \widehat f_{\mathrm{sp}}(r)
 =\sin^2\!\left(\frac{\pi r^2}{2}\right)
   \int_0^\infty B(t)\e^{-\pi r^2t}\dd t,
\end{equation}
where $B$ is continuous, nonnegative, and not identically zero.  The same
argument shows that the integral is strictly positive throughout its
convergence range.  It follows that $\widehat f_{\mathrm{sp}}$ is strictly
positive away from the sine-square zeros for $r>\sqrt2$, and that its zeros at
$r=\sqrt{2n}$, $n\geq2$, are exactly double.

It remains only to cover $0\leq r\leq\sqrt2$, where
\eqref{eq:fsp-fourier-integral} does not converge throughout the interval.
The continuation argument in \cite{CKMRV2017} proves, for
$0<r<\sqrt2$, the exact identity
\begin{align}
 \widehat f_{\mathrm{sp}}(r)
 ={}&\sin^2\!\left(\frac{\pi r^2}{2}\right)\Biggl[
 \int_0^1 B(t)\e^{-\pi r^2t}\dd t \notag\\
 &+\int_1^\infty\left(
 B(t)-\frac{1}{39}t\e^{2\pi t}
 +\frac{10}{117\pi}\e^{2\pi t}\right)
 \e^{-\pi r^2t}\dd t+P(r)\Biggr],
 \label{eq:fsp-low-frequency-continuation}
\end{align}
where
\begin{equation*}
 P(r):=\frac{(10-3\pi)(2-r^2)+3}{117\pi^2(r^2-2)^2}
 \e^{-\pi(r^2-2)}.
\end{equation*}
The inequalities proved there give $B(t)\geq0$ for $t>0$ and
$B(t)-t\e^{2\pi t}/39+10\e^{2\pi t}/(117\pi)\geq0$ for $t\geq1$.
Thus both integrals in
\eqref{eq:fsp-low-frequency-continuation} are nonnegative.  The remaining
term $P(r)$ is strictly positive.  Indeed, polynomial division gives
$\frac{x^4(1-x)^4}{1+x^2}=x^6-4x^5+5x^4-4x^2+4-\frac4{1+x^2}$.
Integration over $[0,1]$ yields
$0<\int_0^1x^4(1-x)^4/(1+x^2)\dd x=22/7-\pi$; hence
$\pi<22/7<10/3$, and therefore $10-3\pi>0$.  Since also
$2-r^2>0$ and $\sin^2(\pi r^2/2)>0$ for $0<r<\sqrt2$,
\eqref{eq:fsp-low-frequency-continuation} gives
$\widehat f_{\mathrm{sp}}(r)>0$ throughout this interval.  Finally,
$\widehat f_{\mathrm{sp}}(0)=1$ and
$\widehat f_{\mathrm{sp}}(\sqrt2)=1/156$.  Hence
$\widehat f_{\mathrm{sp}}(r)>0$ for every $r\geq0$ except the prescribed
zeros $\sqrt{2n}$, $n\geq2$.  Multiplication by $-16380$ proves
\eqref{eq:hatb2-negative}.
\end{proof}

\section{The affine reduction}

The nodal data \eqref{eq:a-data}--\eqref{eq:b-data} immediately yield
\begin{align}
 g_C(2)&=1, \label{eq:gC-at-2}\\
 g_C(\sqrt{2n})=g_C'(\sqrt{2n})&=0 &&(n\geq3),
 \label{eq:gC-direct-nodes}\\
 \widehat g_C(\sqrt{2n})=\widehat g_C'(\sqrt{2n})&=0 &&(n\geq2).
 \label{eq:gC-fourier-nodes}
\end{align}
Thus all interpolation constraints needed by the Leech lattice are automatic.
Only the signs between the nodes remain.

\begin{lemma}[Removable quotient singularities]\label{lem:quotient-extension}
The quotients in \eqref{eq:R-definitions} extend uniquely to smooth real-valued
functions on $[\sqrt6,\infty)$ and $[0,\infty)$, respectively, with one-sided
smoothness at the left endpoints.
\end{lemma}

\begin{proof}
Fix a direct-side node $r_n=\sqrt{2n}$ with $n\geq3$.  By
\eqref{eq:a-data}, $a_2(r_n)=a_2'(r_n)=0$, while
\cref{prop:sign-anchor} says that $b_2$ has an exact double zero at $r_n$.
The smooth Hadamard lemma therefore gives smooth functions $A_n,B_n$ near
$r_n$ such that
\begin{equation}\label{eq:Hadamard-direct}
 a_2(r)=(r-r_n)^2A_n(r),\qquad
 b_2(r)=(r-r_n)^2B_n(r),\qquad B_n(r_n)\neq0.
\end{equation}
Hence $a_2/b_2=A_n/B_n$ extends smoothly across $r_n$.  There are no other
zeros of $b_2$ on the direct-side domain, and the nodes have no finite
accumulation point, so these local extensions patch uniquely.

The same argument applies on the Fourier side: \eqref{eq:a-data} gives
$\widehat a_2(r_n)=\widehat a_2'(r_n)=0$ for $n\geq2$, whereas
\cref{prop:sign-anchor} gives an exact double zero of $\widehat b_2$ and no
additional denominator zeros on $[0,\infty)$.  At $r=0$ the denominator is
nonzero.  Thus the Fourier quotient has a unique smooth real-valued extension.
\end{proof}

The explicit integral kernels used later also make the extension values
transparent.  For $r^2>4$, define
\begin{equation}\label{eq:I-alpha-beta}
 I_\alpha(r):=\int_0^\infty\alpha_2(\ii t)\e^{-\pi r^2t}\dd t,
 \qquad
 I_\beta(r):=\int_0^\infty\beta_2(\ii t)\e^{-\pi r^2t}\dd t,
\end{equation}
and, for $r^2>2$, define
\begin{equation}\label{eq:I-tilded-alpha-beta}
 \widetilde I_\alpha(r):=\int_0^\infty\widetilde\alpha_2(\ii t)
 \e^{-\pi r^2t}\dd t,
 \qquad
 \widetilde I_\beta(r):=\int_0^\infty\widetilde\beta_2(\ii t)
 \e^{-\pi r^2t}\dd t.
\end{equation}
The cusp estimates and large-$t$ bounds established in \cref{sec:laplace}
justify dominated convergence, and hence these four integrals are smooth in
$r$ on their stated domains.  The integral formulas proved there contain the local factor
$4\sin^2(\pi r^2/2)=4\pi^2r_n^2(r-r_n)^2+\Oh((r-r_n)^3)$.
Consequently, whenever $n\geq3$ on the direct side and $n\geq2$ on the
Fourier side,
\begin{align}
 a_2''(r_n)&=8\pi^2r_n^2 I_\alpha(r_n),
 &b_2''(r_n)&=8\pi^2r_n^2 I_\beta(r_n)\neq0,\notag\\
 R(r_n)&=\frac{I_\alpha(r_n)}{I_\beta(r_n)}
       =\frac{a_2''(r_n)}{b_2''(r_n)},
 \label{eq:R-node-value}\\
 \widehat a_2''(r_n)&=8\pi^2r_n^2\widetilde I_\alpha(r_n),
 &\widehat b_2''(r_n)&=8\pi^2r_n^2\widetilde I_\beta(r_n)\neq0,\notag\\
 \widehat R(r_n)&=\frac{\widetilde I_\alpha(r_n)}{\widetilde I_\beta(r_n)}
 =\frac{\widehat a_2''(r_n)}{\widehat b_2''(r_n)}.
 \label{eq:Rhat-node-value}
\end{align}

\begin{proposition}[Sign reduction]\label{prop:sign-reduction}
Assume that $R$ and $\widehat R$ are bounded above on their respective domains.
Then $C_*$ in \eqref{eq:Cstar-definition} is finite, and for every $C\geq C_*$,
\begin{align}
 g_C(r)&\leq0 &&(r\geq\sqrt6),\label{eq:reduced-direct-sign}\\
 \widehat g_C(r)&\geq0 &&(r\geq0).
 \label{eq:reduced-fourier-sign}
\end{align}
Conversely, if $C<C_*$, at least one of these sign conditions fails.
\end{proposition}

\begin{proof}
Away from interpolation nodes, $g_C(r)=b_2(r)(R(r)-C)$.  On
$r\geq\sqrt6$, \cref{prop:sign-anchor} gives $b_2(r)>0$ away from the
nodes.  Thus $C\geq\sup R$ implies \eqref{eq:reduced-direct-sign}; at each
node $g_C$ vanishes by \eqref{eq:gC-direct-nodes}.

Similarly,
$\widehat g_C(r)=\widehat b_2(r)(\widehat R(r)-C)$, and
$\widehat b_2(r)<0$ away from its nodes.  Hence
$C\geq\sup\widehat R$ implies \eqref{eq:reduced-fourier-sign}.

If $C<C_*$, then either $R(r_0)>C$ at some point of the direct domain or
$\widehat R(r_0)>C$ at some point of the Fourier domain.  If $r_0$ is a node,
continuity gives the same strict inequality at nearby nonnodes.  The
corresponding factorization then gives the sign violation.
\end{proof}

It remains to prove boundedness.  Compact subsets are handled by
\cref{lem:quotient-extension}; the sole issue is $r\to\infty$.  The next
sections compute the common limit and its first correction.

\section{Exact extraction of the \texorpdfstring{$n=2$}{n=2} kernel coefficients}
\label{sec:kernel}

\subsection{Modular notation}

Write $q=\e^{2\pi\ii z}$ and $x=\e^{\pi\ii z}$, and let
\begin{align}
 E_2(z)&=1-24\sum_{n\geq1}\sigma_1(n)q^n,\label{eq:E2}\\
 E_4(z)&=1+240\sum_{n\geq1}\sigma_3(n)q^n,\label{eq:E4}\\
 E_6(z)&=1-504\sum_{n\geq1}\sigma_5(n)q^n,\label{eq:E6}\\
 \Delta(z)&=\frac{E_4(z)^3-E_6(z)^2}{1728},
 &j(z)&=\frac{E_4(z)^3}{\Delta(z)}.\label{eq:Delta-j}
\end{align}
We also use $U=\theta_3^4$, $V=\theta_2^4$, and $W=\theta_4^4$; then
$U=V+W$ and $\Delta=(UVW)^2/256$.  Set $\lambda=V/U$,
$L=\log\lambda$, and $L_S=\log(1-\lambda)=\log(W/U)$, with the branches
used in \cite{CKMRV2022}.  On the positive imaginary axis,
$L_S$ is the ordinary real logarithm of $W/U$.

For even $k$, let
\begin{equation}\label{eq:f-list}
 (f_{-2},f_0,f_2,f_4,f_6,f_8,f_{10},f_{12},f_{14})
 =\left(\frac{E_{10}}{\Delta},1,\frac{E_{14}}{\Delta},
 E_4,E_6,E_8,E_{10},\Delta,E_{14}\right),
\end{equation}
where $E_8=E_4^2$, $E_{10}=E_4E_6$, and $E_{14}=E_4^2E_6$, and put
$\Pi_{a,b,c}:=\diag(f_a,f_b,f_c)$.  Let $N=1728$ and
$j_{\tau,z}=j(\tau)-j(z)$.

The three plus-basis functions in the $\tau$ variable are
\begin{equation}\label{eq:plus-basis-tau}
 \varphi_{-2}(\tau)=\tau,
 \qquad
 \varphi_0(\tau)=\tau E_2(\tau)-\frac{3\ii}{\pi},
 \qquad
 \varphi_2(\tau)=\tau E_2(\tau)^2-\frac{6\ii}{\pi}E_2(\tau),
\end{equation}
and in the $z$ variable
\begin{equation}\label{eq:plus-basis-z}
 \widetilde\varphi_{-2}(z)=z^2,
 \quad
 \widetilde\varphi_0(z)=z^2(E_2|_2S)(z),
 \quad
 \widetilde\varphi_2(z)=z^2(E_2|_2S)(z)^2,
\end{equation}
where $S:z\mapsto-1/z$ and $(F|_kS)(z)=z^{-k}F(-1/z)$.

The minus-basis functions are
\begin{align}
 \psi_0&=1,\label{eq:psi0}\\
 \psi_2&=(U+W)L+(-U-V)L_S,\label{eq:psi2}\\
 \psi_4&=(U^2+W^2-2V^2)L+(U^2+V^2-2W^2)L_S.\label{eq:psi4}
\end{align}
The corresponding tilded functions are $\widetilde\psi_0=L$,
$\widetilde\psi_2=W$, and $\widetilde\psi_4=U^2-V^2$.

\subsection{The generating kernels}

For $d=24$, the explicit coefficient matrices in \cite{CKMRV2022} are
\begin{align}
 \Upsilon_+^{(24)}(\tau,z)
 ={}&\frac{\Pi_{14,12,10}(\tau)}{36N\pi^{-2}\ii\Delta(z)}
 \begin{pmatrix}
 6&0&N/j_{\tau,z}-6\\
 -12j(\tau)+5N&-2N/j_{\tau,z}&12j(\tau)-7N\\
 N/j_{\tau,z}+6&0&-6
 \end{pmatrix}
 \Pi_{4,2,0}(z),
 \label{eq:Upsilon-plus}\\[3mm]
 \Upsilon_-^{(24)}(\tau,z)
 ={}&\frac{\Pi_{12,10,8}(\tau)}{2N\pi\Delta(z)j_{\tau,z}}
 \begin{pmatrix}
 -2N&-2Nj_{\tau,z}&0\\
 0&j(\tau)+2j_{\tau,z}&-1\\
 0&N-2j_{\tau,z}-j(\tau)&1
 \end{pmatrix}
 \Pi_{2,0,-2}(z).
 \label{eq:Upsilon-minus}
\end{align}
Define
\begin{align}
 K_+(\tau,z)
 &=(\varphi_{-2},\varphi_0,\varphi_2)(\tau)
   \Upsilon_+^{(24)}(\tau,z)
   \begin{pmatrix}\widetilde\varphi_{-2}\\
                   \widetilde\varphi_0\\
                   \widetilde\varphi_2\end{pmatrix}(z),
 \label{eq:K-plus}\\
 K_-(\tau,z)
 &=(\psi_0,\psi_2,\psi_4)(\tau)
   \Upsilon_-^{(24)}(\tau,z)
   \begin{pmatrix}\widetilde\psi_0\\
                   \widetilde\psi_2\\
                   \widetilde\psi_4\end{pmatrix}(z),
 \label{eq:K-minus}\\
 K&=\frac{K_++K_-}{2},
 &\widehat K&=\frac{K_+-K_-}{2}.
 \label{eq:K-and-Khat}
\end{align}
Their $\tau$-expansions have the form
\begin{align}
 K(\tau,z)
 &=\sum_{n\geq2}\alpha_n(z)\e^{2\pi\ii n\tau}
 +2\pi\ii\tau\sum_{n\geq2}\sqrt{2n}\,\beta_n(z)
   \e^{2\pi\ii n\tau},
 \label{eq:K-expansion}\\
 \widehat K(\tau,z)
 &=\sum_{n\geq2}\widetilde\alpha_n(z)\e^{2\pi\ii n\tau}
 +2\pi\ii\tau\sum_{n\geq2}\sqrt{2n}\,\widetilde\beta_n(z)
   \e^{2\pi\ii n\tau}.
 \label{eq:Khat-expansion}
\end{align}
For $n=2$, the coefficient of $\e^{4\pi\ii\tau}$ in
\eqref{eq:K-expansion} is $\alpha_2(z)+4\pi\ii\tau\,\beta_2(z)$, and
analogously for the tilded pair.

\subsection{Finite \texorpdfstring{$\tau$}{tau}-side algebra}

Put $X=\e^{\pi\ii\tau}$ and $T=\pi\ii\tau$.  For the coefficient extraction
below, we retain the level-two forms through the orders displayed:
\begin{align}
 U(\tau)
 &=1+8X+24X^2+32X^3+24X^4+48X^5+96X^6
   +64X^7+\Oh(X^8),\label{eq:U-tau-series}\\
 V(\tau)
 &=16X+64X^3+96X^5+128X^7+\Oh(X^9),
 \label{eq:V-tau-series}\\
 W(\tau)
 &=1-8X+24X^2-32X^3+24X^4-48X^5+96X^6
   -64X^7+\Oh(X^8).
 \label{eq:W-tau-series}
\end{align}
The level-one series are
\begin{align}
 E_2(\tau)&=1-24X^2-72X^4-96X^6+\Oh(X^8),
 \label{eq:E2-tau-series}\\
 E_4(\tau)&=1+240X^2+2160X^4+6720X^6+\Oh(X^8),
 \label{eq:E4-tau-series}\\
 E_6(\tau)&=1-504X^2-16632X^4-122976X^6+\Oh(X^8),
 \label{eq:E6-tau-series}\\
 E_8(\tau)&=1+480X^2+61920X^4+1050240X^6+\Oh(X^8),
 \label{eq:E8-tau-series}\\
 E_{10}(\tau)&=1-264X^2-135432X^4-5196576X^6+\Oh(X^8),
 \label{eq:E10-tau-series}\\
 E_{14}(\tau)&=1-24X^2-196632X^4-38263776X^6+\Oh(X^8),
 \label{eq:E14-tau-series}\\
 \Delta(\tau)&=X^2-24X^4+252X^6+\Oh(X^8),
 \label{eq:Delta-tau-series}\\
 j(\tau)&=X^{-2}+744+196884X^2+21493760X^4
 +864299970X^6+\Oh(X^8).
 \label{eq:j-tau-series}
\end{align}
The coefficient $864299970$ is fixed by the identity $\Delta j=E_4^3$.
To determine it one must retain
$\Delta=X^2-24X^4+252X^6-1472X^8+4830X^{10}+\Oh(X^{12})$; in particular,
the term $4830X^{10}$ contributes to $[X^6]j$.  The logarithmic functions have
the branch-sensitive expansions
\begin{align}
 L(\tau)
 ={}&T+4\log2-8X+12X^2-\frac{32}{3}X^3+6X^4
 -\frac{48}{5}X^5+16X^6+\Oh(X^7),
 \label{eq:L-tau-series}\\
 L_S(\tau)
 ={}&-16X-\frac{64}{3}X^3-\frac{96}{5}X^5
 -\frac{128}{7}X^7+\Oh(X^9).
 \label{eq:LS-tau-series}
\end{align}
In particular, formal inversion gives
\begin{equation}\label{eq:inverse-jtauz-short}
 \frac1{j(\tau)-J}
 =X^2+(J-744)X^4+(J^2-1488J+356652)X^6+\Oh(X^8).
\end{equation}
The minus-basis combinations simplify to even series:
\begin{align}
 \psi_2(\tau)
 ={}&2T+8\log2
 +(48T+408+192\log2)X^2 \notag\\
 &+(48T+2636+192\log2)X^4+\Oh(X^6),
 \label{eq:psi2-short-series}\\
 \psi_4(\tau)
 ={}&2T+8\log2
 +(-288T-744-1152\log2)X^2 \notag\\
 &+(-1824T-25972-7296\log2)X^4+\Oh(X^6).
 \label{eq:psi4-short-series}
\end{align}
Equations \eqref{eq:plus-basis-tau} and \eqref{eq:E2-tau-series} show that
every plus-basis entry is affine in $T$.
Together with \eqref{eq:psi2-short-series}--\eqref{eq:psi4-short-series},
this proves that the $X^4$ coefficients of both $K_+$ and $K_-$, and hence of
$K$ and $\widehat K$, are affine in $T$, as required by the coefficient
normalization stated above.

These truncations are sufficient for a valuation reason.  Every occurrence of
$(j(\tau)-J)^{-1}$ begins with $X^2$, while the sole negative power among the
remaining $\tau$-dependent factors is the leading term $X^{-2}$ of $j(\tau)$.
Consequently an input term of degree greater than $6$ cannot contribute to the
coefficient of $X^4$ in either matrix product in
\eqref{eq:Upsilon-plus}--\eqref{eq:Upsilon-minus}.  Thus the extraction below
is an identity in truncated formal power series, rather than an asymptotic
approximation.

Introduce the $z$-dependent abbreviations $D=\Delta(z)$, $J=j(z)$,
$P_{-2}=z^2$, $P_0=z^2(E_2|_2S)(z)$, and
$P_2=z^2(E_2|_2S)(z)^2$, together with $S_0=L(z)$, $S_2=W(z)$, and
$S_4=U(z)^2-V(z)^2$.

\begin{proposition}[Exact $n=2$ coefficient extraction]\label{prop:kernel-coefficients}
The four coefficient functions in
\eqref{eq:K-expansion}--\eqref{eq:Khat-expansion} are
\begin{align}
\alpha_2(z)={}&-\frac{1}{12\pi D^2}\Bigl[
 \pi^2DE_4JP_{-2}-2163\pi^2DE_4P_{-2}
 -16DJS_2\log2-2DJS_2 \notag\\
&\hspace{12mm}+1155\pi^2DP_2+640DS_2+15360DS_2\log2
 -8E_{10}S_4\log2+2E_{10}S_4 \notag\\
&\hspace{12mm}-\pi^2E_{14}P_0+6E_{14}S_0\Bigr],
\label{eq:alpha2-explicit}\\[1mm]
\beta_2(z)={}&-\frac{1}{288\pi D^2}\Bigl[
 \pi^2DE_4JP_{-2}-3528\pi^2DE_4P_{-2}
 +\pi^2DJP_2-24DJS_2 \notag\\
&\hspace{12mm}+1800\pi^2DP_2+23040DS_2
 -12E_{10}S_4-2\pi^2E_{14}P_0\Bigr],
\label{eq:beta2-explicit}\\[1mm]
\widetilde\alpha_2(z)={}&-\frac{1}{12\pi D^2}\Bigl[
 \pi^2DE_4JP_{-2}-2163\pi^2DE_4P_{-2}
 +2DJS_2+16DJS_2\log2 \notag\\
&\hspace{12mm}+1155\pi^2DP_2-15360DS_2\log2-640DS_2
 -2E_{10}S_4+8E_{10}S_4\log2 \notag\\
&\hspace{12mm}-\pi^2E_{14}P_0-6E_{14}S_0\Bigr],
\label{eq:alphatilde2-explicit}\\[1mm]
\widetilde\beta_2(z)={}&-\frac{1}{288\pi D^2}\Bigl[
 \pi^2DE_4JP_{-2}-3528\pi^2DE_4P_{-2}
 +\pi^2DJP_2+24DJS_2 \notag\\
&\hspace{12mm}+1800\pi^2DP_2-23040DS_2
 +12E_{10}S_4-2\pi^2E_{14}P_0\Bigr].
\label{eq:betatilde2-explicit}
\end{align}
Here every modular form in the brackets is evaluated at $z$.
\end{proposition}

\begin{proof}
All calculations take place in the ring of truncated Laurent series in $X$
with coefficients in
$\C(T,\pi,\log2,D,J,E_4,E_{10},E_{14},P_{-2},P_0,P_2,S_0,S_2,S_4)$.
The geometric identity
$1/(j(\tau)-j(z))=\sum_{m\geq0}j(z)^m/j(\tau)^{m+1}$ is valid
formally because $j(\tau)^{-1}\in X^2\C[[X^2]]$; its truncation through
degree $6$ is precisely \eqref{eq:inverse-jtauz-short}.

To make the coefficient collection explicit, write
$[X^4]K_+(\tau,z)=A_+(z)+4TB_+(z)$ and
$[X^4]K_-(\tau,z)=A_-(z)+4TB_-(z)$.  Substitution of
\eqref{eq:U-tau-series}--\eqref{eq:psi4-short-series} into the three rows of \eqref{eq:Upsilon-plus} gives
\begin{align}
 A_+={}&-\frac{\pi}{6D^2}\bigl(
 DE_4JP_{-2}-2163DE_4P_{-2}+1155DP_2-E_{14}P_0\bigr),
 \label{eq:Aplus}\\
 B_+={}&-\frac{\pi}{144D^2}\bigl(
 DE_4JP_{-2}-3528DE_4P_{-2}+DJP_2+1800DP_2-2E_{14}P_0\bigr).
 \label{eq:Bplus}
\end{align}
Likewise, substitution into \eqref{eq:Upsilon-minus} gives
\begin{align}
 A_-={}&\frac{1}{3\pi D^2}\bigl(
 8DJS_2\log2+DJS_2-7680DS_2\log2-320DS_2 \notag\\
 &\hspace{34mm}+4E_{10}S_4\log2-E_{10}S_4-3E_{14}S_0\bigr),
 \label{eq:Aminus}\\
 B_-={}&\frac{1}{12\pi D^2}
 \bigl(2DJS_2-1920DS_2+E_{10}S_4\bigr).
 \label{eq:Bminus}
\end{align}
These four identities are obtained by ordinary multiplication of the displayed
finite series; no omitted term can enter by the valuation argument preceding
the proposition.

Since $K=(K_++K_-)/2$ and $\widehat K=(K_+-K_-)/2$, coefficient comparison
gives $\alpha_2=(A_++A_-)/2$, $\beta_2=(B_++B_-)/2$,
$\widetilde\alpha_2=(A_+-A_-)/2$, and
$\widetilde\beta_2=(B_+-B_-)/2$.  Putting
\eqref{eq:Aplus}--\eqref{eq:Bminus} over common denominators gives
\eqref{eq:alpha2-explicit}--\eqref{eq:betatilde2-explicit} term by term.
\end{proof}

\section{The \texorpdfstring{$S$}{S}-cusp calculation}
\label{sec:cusp}

\subsection{Transformation to the cusp at zero}

Let $z=\ii t$, $w=-1/z=\ii/t$, and
$x=\e^{\pi\ii w}=\e^{-\pi/t}$.  The level-one transformations are
\begin{equation}\label{eq:S-transform-level-one}
 \Delta(z)=z^{-12}\Delta(w),
 \qquad
 E_k(z)=z^{-k}E_k(w)
 \quad(k=4,6,8,10,14),
 \qquad
 j(z)=j(w),
\end{equation}
and the $z$-basis factors become
\begin{align}
 P_{-2}(z)&=z^2,
 &P_0(z)&=E_2(w),
 &P_2(z)&=z^{-2}E_2(w)^2,
 \label{eq:P-transform}\\
 S_0(z)&=L_S(w),
 &S_2(z)&=-z^{-2}V(w),
 &S_4(z)&=z^{-4}\bigl(U(w)^2-W(w)^2\bigr).
 \label{eq:S-transform}
\end{align}
The quasimodular correction in $P_0,P_2$ cancels because those quantities were
defined using $E_2|_2S$ rather than $E_2$ itself.

After factoring the common weight $z^{-14}$ from the brackets in
\eqref{eq:alpha2-explicit} and \eqref{eq:beta2-explicit}, define
\begin{align}
\mathcal A(w):={}&
 \pi^2\Delta E_4j-2163\pi^2\Delta E_4
 +16\Delta jV\log2+2\Delta jV
 +1155\pi^2\Delta E_2^2 \notag\\
&-640\Delta V-15360\Delta V\log2
 -8E_{10}(U^2-W^2)\log2+2E_{10}(U^2-W^2) \notag\\
&-\pi^2E_{14}E_2+6E_{14}L_S,
\label{eq:calA}\\[1mm]
\mathcal B(w):={}&
 \pi^2\Delta E_4j-3528\pi^2\Delta E_4
 +\pi^2\Delta jE_2^2+24\Delta jV
 +1800\pi^2\Delta E_2^2 \notag\\
&-23040\Delta V-12E_{10}(U^2-W^2)-2\pi^2E_{14}E_2.
\label{eq:calB}
\end{align}
Every function on the right is evaluated at $w$.  Then
\begin{align}
 \alpha_2(\ii t)
 &=-\frac{z^{10}}{12\pi}\frac{\mathcal A(w)}{\Delta(w)^2},
 &\beta_2(\ii t)
 &=-\frac{z^{10}}{288\pi}\frac{\mathcal B(w)}{\Delta(w)^2}.
 \label{eq:alpha-beta-transformed}
\end{align}

Near $x=0$, the convergent series of the level-one forms are even, whereas
$V(-x)=-V(x)$, $(U^2-W^2)(-x)=-(U^2-W^2)(x)$, and
$L_S(-x)=-L_S(x)$.  Comparison with \eqref{eq:alphatilde2-explicit} and
\eqref{eq:betatilde2-explicit} gives the exact parity identities
\begin{align}
 \widetilde\alpha_2(\ii t)
 &=-\frac{z^{10}}{12\pi}
   \frac{\mathcal A(w)|_{x\mapsto-x}}{\Delta(w)^2},
 \label{eq:alphatilde-parity}\\
 \widetilde\beta_2(\ii t)
 &=-\frac{z^{10}}{288\pi}
   \frac{\mathcal B(w)|_{x\mapsto-x}}{\Delta(w)^2}.
 \label{eq:betatilde-parity}
\end{align}

\subsection{Finite cancellation at the cusp}

The required elementary expansions are
\begin{align}
 \Delta(w)&=x^2-24x^4+252x^6+\Oh(x^8),
 \label{eq:Delta-x}\\
 j(w)&=x^{-2}+744+196884x^2+21493760x^4
       +864299970x^6+\Oh(x^8),
 \label{eq:j-x}\\
 L_S(w)&=-16x-\frac{64}{3}x^3-\frac{96}{5}x^5
          -\frac{128}{7}x^7+\Oh(x^9).
 \label{eq:LS-x}
\end{align}
The last identity follows from $L_S=\log(W/U)$ and the theta expansions.
Before truncating \eqref{eq:calA} and \eqref{eq:calB}, we use the exact
identity $\Delta j=E_4^3$ to make the replacements
$\Delta E_4j=E_4^4$, $\Delta jV=E_4^3V$, and
$\Delta jE_2^2=E_4^3E_2^2$.  Every summand is then a convergent power series
in $x$ at the origin, and the series displayed in
\eqref{eq:U-tau-series}--\eqref{eq:E14-tau-series} (with $\tau$ replaced by
$w$), together with \eqref{eq:Delta-x} and \eqref{eq:LS-x}, determine all
coefficients through degree $6$.

For completeness, \cref{tab:A-cancellation,tab:B-cancellation} record every
summand through that degree.  Blank entries are zero.

\begin{table}[!htbp]
\centering
\small
\caption{Coefficients $[x^k]$ of the summands in $\mathcal A$, $0\leq k\leq6$.}
\label{tab:A-cancellation}
\resizebox{\textwidth}{!}{%
\begin{tabular}{@{}lrrrrrrr@{}}
\toprule
summand & $x^0$ & $x^1$ & $x^2$ & $x^3$ & $x^4$ & $x^5$ & $x^6$\\
\midrule
$\pi^2\Delta E_4j$
&$\pi^2$&&$960\pi^2$&&$354240\pi^2$&&$61543680\pi^2$\\
$-2163\pi^2\Delta E_4$
&&&$-2163\pi^2$&&$-467208\pi^2$&&$7241724\pi^2$\\
$16\Delta jV\log2$
&&$256\log2$&&$185344\log2$&&$46634496\log2$&\\
$2\Delta jV$
&&$32$&&$23168$&&$5829312$&\\
$1155\pi^2\Delta E_2^2$
&&&$1155\pi^2$&&$-83160\pi^2$&&$2120580\pi^2$\\
$-640\Delta V$
&&&&$-10240$&&$204800$&\\
$-15360\Delta V\log2$
&&&&$-245760\log2$&&$4915200\log2$&\\
$-8E_{10}(U^2-W^2)\log2$
&&$-256\log2$&&$60416\log2$&&$36530688\log2$&\\
$2E_{10}(U^2-W^2)$
&&$64$&&$-15104$&&$-9132672$&\\
$-\pi^2E_{14}E_2$
&$-\pi^2$&&$48\pi^2$&&$196128\pi^2$&&$33542976\pi^2$\\
$6E_{14}L_S$
&&$-96$&&$2176$&&$94398144/5$&\\
\midrule
sum &0&0&0&0&0&$\frac{1835008}{5}(43+240\log2)$
&$104448960\pi^2$\\
\bottomrule
\end{tabular}%
}
\end{table}

\begin{table}[!htbp]
\centering
\small
\caption{Coefficients $[x^k]$ of the summands in $\mathcal B$, $0\leq k\leq6$.}
\label{tab:B-cancellation}
\resizebox{\textwidth}{!}{%
\begin{tabular}{@{}lrrrrrrr@{}}
\toprule
summand & $x^0$ & $x^1$ & $x^2$ & $x^3$ & $x^4$ & $x^5$ & $x^6$\\
\midrule
$\pi^2\Delta E_4j$
&$\pi^2$&&$960\pi^2$&&$354240\pi^2$&&$61543680\pi^2$\\
$-3528\pi^2\Delta E_4$
&&&$-3528\pi^2$&&$-762048\pi^2$&&$11811744\pi^2$\\
$\pi^2\Delta jE_2^2$
&$\pi^2$&&$672\pi^2$&&$145152\pi^2$&&$8663424\pi^2$\\
$24\Delta jV$
&&$384$&&$278016$&&$69951744$&\\
$1800\pi^2\Delta E_2^2$
&&&$1800\pi^2$&&$-129600\pi^2$&&$3304800\pi^2$\\
$-23040\Delta V$
&&&&$-368640$&&$7372800$&\\
$-12E_{10}(U^2-W^2)$
&&$-384$&&$90624$&&$54796032$&\\
$-2\pi^2E_{14}E_2$
&$-2\pi^2$&&$96\pi^2$&&$392256\pi^2$&&$67085952\pi^2$\\
\midrule
sum &0&0&0&0&0&$132120576$&$152409600\pi^2$\\
\bottomrule
\end{tabular}%
}
\end{table}
\FloatBarrier

Thus $[x^6]\mathcal A=104448960\pi^2$ and
$[x^6]\mathcal B=152409600\pi^2$, and the two series begin as
\begin{align}
 \mathcal A(w)
 &=\frac{1835008}{5}(43+240\log2)x^5
   +104448960\pi^2x^6+\Oh(x^7),
 \label{eq:A-leading}\\
 \mathcal B(w)
 &=132120576x^5+152409600\pi^2x^6+\Oh(x^7).
 \label{eq:B-leading}
\end{align}

\begin{proposition}[Coefficient functions at the $S$-cusp]\label{prop:cusp-asymptotics}
Set $c_0:=458752/\pi$ and $\rho:=(43+240\log2)/15$.  As $t\to0^+$,
\begin{align}
 \alpha_2(\ii t)
 &=\rho c_0t^{10}\e^{-\pi/t}
   +\Oh\!\left(t^{10}\e^{-2\pi/t}\right),
 \label{eq:alpha-cusp}\\
 \beta_2(\ii t)
 &=c_0t^{10}\e^{-\pi/t}
   +\Oh\!\left(t^{10}\e^{-2\pi/t}\right),
 \label{eq:beta-cusp}\\
 \widetilde\alpha_2(\ii t)
 &=-\rho c_0t^{10}\e^{-\pi/t}
   +\Oh\!\left(t^{10}\e^{-2\pi/t}\right),
 \label{eq:alphatilde-cusp}\\
 \widetilde\beta_2(\ii t)
 &=-c_0t^{10}\e^{-\pi/t}
   +\Oh\!\left(t^{10}\e^{-2\pi/t}\right).
 \label{eq:betatilde-cusp}
\end{align}
More precisely, with $d_1:=10080\pi(713-840\log2)>0$, one has
\begin{align}
 \alpha_2(\ii t)-\rho\beta_2(\ii t)
 &=d_1t^{10}\e^{-2\pi/t}
   +\Oh\!\left(t^{10}\e^{-3\pi/t}\right),
 \label{eq:direct-difference-cusp}\\
 \widetilde\alpha_2(\ii t)-\rho\widetilde\beta_2(\ii t)
 &=d_1t^{10}\e^{-2\pi/t}
   +\Oh\!\left(t^{10}\e^{-3\pi/t}\right).
 \label{eq:fourier-difference-cusp}
\end{align}
\end{proposition}

\begin{proof}
Write $\Delta(w)=x^2d(x)$, where
$d(x)=1-24x^2+252x^4+\Oh(x^6)$.  Since $d(0)=1$, formal inversion gives
$\Delta(w)^{-2}=x^{-4}(1+48x^2+\Oh(x^4))$.  Moreover, $z^{10}=(\ii t)^{10}=-t^{10}$.  Substitution of
\eqref{eq:A-leading} and \eqref{eq:B-leading} into
\eqref{eq:alpha-beta-transformed}, followed by this inverse-square expansion,
therefore gives the first two asymptotic
formulas.  Indeed, the leading constants are
$\frac{1}{12\pi}\frac{1835008}{5}(43+240\log2)=\rho c_0$ and
$132120576/(288\pi)=c_0$.  The parity identities
\eqref{eq:alphatilde-parity}--\eqref{eq:betatilde-parity} reverse the odd
$x^5$ coefficients and leave the even $x^6$ coefficients unchanged, which
proves the two tilded formulas.

For the refined statement, set
$F(x):=\mathcal A(w)/(12\pi)-\rho\mathcal B(w)/(288\pi)$.  The
coefficient of $x^5$ in $F$ is zero by the definition of $\rho$.  By
\cref{tab:A-cancellation,tab:B-cancellation}, its coefficient of $x^6$ is
\begin{align}
 \frac{104448960\pi^2}{12\pi}
 -\rho\frac{152409600\pi^2}{288\pi}
 &=\pi(8704080-529200\rho)\notag\\
 &=10080\pi(713-840\log2)=d_1.
 \label{eq:d1-calculation}
\end{align}
Because every transformed summand is analytic at $x=0$, it follows that
$F(x)=d_1x^6+\Oh(x^7)$.  The transformed formulas and the inverse-square expansion now give
\begin{align*}
 \alpha_2(\ii t)-\rho\beta_2(\ii t)
 &=t^{10}x^{-4}\bigl(1+48x^2+\Oh(x^4)\bigr)
   \bigl(d_1x^6+\Oh(x^7)\bigr)\\
 &=d_1t^{10}x^2+\Oh(t^{10}x^3).
\end{align*}
For the tilded difference the corresponding series is $F(-x)$.  Its $x^5$
coefficient is again zero and its $x^6$ coefficient is again $d_1$, so the
same calculation proves \eqref{eq:fourier-difference-cusp}.

Finally, $d_1>0$ follows without decimal approximation.  The exponential
series gives
$\e^{7/10}>1+7/10+(7/10)^2/2+(7/10)^3/6=12013/6000>2$; hence
$\log2<7/10$ and $713-840\log2>713-588=125$.
\end{proof}

\section{Laplace--Bessel transfer and boundedness}
\label{sec:laplace}

We now transfer the cusp expansions of \cref{prop:cusp-asymptotics} to the
large-radius behavior of the interpolation basis.  Proposition~5.4 of
\cite{CKMRV2022}, specialized to $d=24$ and $n=2$, gives
\begin{align}
 a_2(r)
 &=4\sin^2\!\left(\frac{\pi r^2}{2}\right)
   \int_0^\infty\alpha_2(\ii t)\e^{-\pi r^2t}\dd t,
 &&r^2>4,\label{eq:a2-Laplace}\\
 b_2(r)
 &=4\sin^2\!\left(\frac{\pi r^2}{2}\right)
   \int_0^\infty\beta_2(\ii t)\e^{-\pi r^2t}\dd t,
 &&r^2>4,\label{eq:b2-Laplace}\\
 \widehat a_2(r)
 &=4\sin^2\!\left(\frac{\pi r^2}{2}\right)
   \int_0^\infty\widetilde\alpha_2(\ii t)\e^{-\pi r^2t}\dd t,
 &&r^2>2,\label{eq:ahat2-Laplace}\\
 \widehat b_2(r)
 &=4\sin^2\!\left(\frac{\pi r^2}{2}\right)
   \int_0^\infty\widetilde\beta_2(\ii t)\e^{-\pi r^2t}\dd t,
 &&r^2>2.\label{eq:bhat2-Laplace}
\end{align}
The direct formulas are therefore valid on the whole sign domain
$r\geq\sqrt6$, and the Fourier formulas are valid in a neighborhood of
infinity.  Compactness and \cref{lem:quotient-extension} handle the remaining
Fourier interval.

For $a>0$ and $r>0$, put
\begin{equation}\label{eq:Ja-definition}
 J_a(r):=\int_0^\infty t^{10}
 \exp\!\left(-\frac{a\pi}{t}-\pi r^2t\right)\dd t.
\end{equation}
The classical integral representation of the modified Bessel function gives
\begin{equation}\label{eq:Ja-Bessel}
 J_a(r)=2a^{11/2}r^{-11}K_{11}(2\pi r\sqrt a).
\end{equation}
Consequently, as $r\to\infty$,
\begin{equation}\label{eq:Ja-asymptotic}
 J_a(r)
 =a^{21/4}r^{-23/2}\e^{-2\pi r\sqrt a}
  \left(1+\Oh_a(r^{-1})\right),
\end{equation}
where the leading constant simplifies exactly to $1$.  In particular,
\begin{equation}\label{eq:J-ratio-asymptotic}
 \frac{J_b(r)}{J_a(r)}
 =\left(\frac ba\right)^{21/4}
  \e^{-2\pi r(\sqrt b-\sqrt a)}
  \left(1+\Oh_{a,b}(r^{-1})\right).
\end{equation}

\begin{lemma}[Cusp-to-Laplace transfer]\label{lem:Laplace-transfer}
Let $u,v:(0,\infty)\to\C$ be continuous.  Suppose that, for some
$A,B\in\C$ with $B\neq0$ and some $\eta>1$,
\begin{align}
 u(t)&=At^{10}\e^{-\pi/t}
       +\Oh\!\left(t^{10}\e^{-\eta\pi/t}\right),\label{eq:u-cusp-general}\\
 v(t)&=Bt^{10}\e^{-\pi/t}
       +\Oh\!\left(t^{10}\e^{-\eta\pi/t}\right)\label{eq:v-cusp-general}
\end{align}
as $t\to0^+$.  Assume also that for some $M,c\geq0$,
\begin{equation}\label{eq:large-t-general-growth}
 |u(t)|+|v(t)|\ll (1+t)^M\e^{c t}\qquad(t\geq1).
\end{equation}
Then, for all sufficiently large $r$, both Laplace integrals converge and
\begin{equation}\label{eq:Laplace-ratio-limit}
 \frac{\int_0^\infty u(t)\e^{-\pi r^2t}\dd t}
      {\int_0^\infty v(t)\e^{-\pi r^2t}\dd t}
 \longrightarrow\frac AB.
\end{equation}
More precisely,
\begin{align}
 \int_0^\infty u(t)\e^{-\pi r^2t}\dd t
 &=AJ_1(r)+\Oh(J_\eta(r)),\label{eq:u-transfer}\\
 \int_0^\infty v(t)\e^{-\pi r^2t}\dd t
 &=BJ_1(r)+\Oh(J_\eta(r)).\label{eq:v-transfer}
\end{align}
If, in addition,
\begin{equation}\label{eq:refined-difference-hypothesis}
 u(t)-\frac ABv(t)
 =Dt^{10}\e^{-\eta\pi/t}
  +\Oh\!\left(t^{10}\e^{-\zeta\pi/t}\right)
\end{equation}
with $D\neq0$ and $\zeta>\eta$, then
\begin{equation}\label{eq:refined-transfer-conclusion}
 \frac{\int_0^\infty u(t)\e^{-\pi r^2t}\dd t}
      {\int_0^\infty v(t)\e^{-\pi r^2t}\dd t}
 -\frac AB
 =\frac DB\frac{J_\eta(r)}{J_1(r)}\bigl(1+o(1)\bigr).
\end{equation}
\end{lemma}

\begin{proof}
Choose $\delta\in(0,1)$ so that the two cusp estimates hold uniformly on
$(0,\delta]$.  On this interval,
\begin{equation}\label{eq:small-t-u-transfer}
 \int_0^\delta u(t)\e^{-\pi r^2t}\dd t
 =A\int_0^\delta t^{10}\e^{-\pi/t-\pi r^2t}\dd t
 +\Oh\!\left(\int_0^\delta t^{10}
 \e^{-\eta\pi/t-\pi r^2t}\dd t\right).
\end{equation}
For fixed $a>0$, the tail of the corresponding model integral satisfies
\begin{align*}
 \int_\delta^\infty t^{10}\e^{-a\pi/t-\pi r^2t}\dd t
 &\leq \e^{-\pi\delta r^2/2}
       \int_0^\infty t^{10}\e^{-\pi r^2t/2}\dd t\\
 &=\Oh_\delta\!\left(r^{-22}\e^{-\pi\delta r^2/2}\right).
\end{align*}
By \eqref{eq:Ja-asymptotic}, this Gaussian bound is $o(J_a(r))$.  Extending
the two model integrals in \eqref{eq:small-t-u-transfer} from $\delta$ to
infinity therefore gives
\begin{equation}\label{eq:small-t-u-model}
 \int_0^\delta u(t)\e^{-\pi r^2t}\dd t
 =AJ_1(r)+\Oh(J_\eta(r)).
\end{equation}

It remains to estimate the actual integral outside $(0,\delta]$.  Continuity
on the compact interval $[\delta,1]$ gives a constant $K_\delta$ such that
\begin{equation}\label{eq:middle-t-tail}
 \int_\delta^1 |u(t)|\e^{-\pi r^2t}\dd t
 \leq K_\delta\e^{-\pi\delta r^2}=o(J_\eta(r)).
\end{equation}
For $t\geq1$, let $C$ be an admissible constant in
\eqref{eq:large-t-general-growth} and put $\lambda=\pi r^2-c$.  Once
$\lambda\geq2$, the elementary domination $(1+t)^M\leq C_M\e^t$ for
$t\geq1$ yields
\begin{equation}\label{eq:large-t-tail}
 \int_1^\infty |u(t)|\e^{-\pi r^2t}\dd t
 \leq CC_M\int_1^\infty \e^{-(\lambda-1)t}\dd t
 =\Oh_{u,M,c}\!\left(\e^{-(\lambda-1)}\right)
 =o(J_\eta(r)).
\end{equation}
The same estimates hold for $v$.  Combining
\eqref{eq:small-t-u-model}--\eqref{eq:large-t-tail} proves
\eqref{eq:u-transfer}--\eqref{eq:v-transfer}; it also proves convergence of
the Laplace integrals for all sufficiently large $r$.  Since
$J_\eta(r)/J_1(r)\to0$, the denominator is $BJ_1(r)(1+o(1))$ and is therefore
nonzero for all sufficiently large $r$.  This proves
\eqref{eq:Laplace-ratio-limit}.

The function $u-(A/B)v$ is continuous and satisfies the same large-$t$ growth
bound.  Repeating the preceding argument with
\eqref{eq:refined-difference-hypothesis} gives
\begin{equation}\label{eq:refined-difference-integral}
 \int_0^\infty\left(u(t)-\frac ABv(t)\right)
 \e^{-\pi r^2t}\dd t
 =DJ_\eta(r)+\Oh(J_\zeta(r)).
\end{equation}
After division by $BJ_1(r)(1+o(1))$, the relation
$J_\zeta(r)=o(J_\eta(r))$ proves
\eqref{eq:refined-transfer-conclusion}.
\end{proof}

\begin{proposition}[Common quotient limit]\label{prop:quotient-limit}
The smooth quotient extensions from \cref{lem:quotient-extension} satisfy
$\lim_{r\to\infty}R(r)=\lim_{r\to\infty}\widehat R(r)=\rho$.  In
particular, they are bounded above on their respective domains, and the number $C_*$ in \eqref{eq:Cstar-definition} is finite.
\end{proposition}

\begin{proof}
For $d=24$, the coefficient expansions in
\cite[Eq.~(5.22)]{CKMRV2022} start at index $m_0=-2$.  Specializing them to
$n=2$ and $z=\ii t$ gives, as $t\to\infty$,
\begin{align}
 \alpha_2(\ii t)
 &=t\e^{4\pi t}+\Oh\!\left((1+t)^2\e^{2\pi t}\right),
 \label{eq:alpha-large-t}\\
 \beta_2(\ii t)
 &=\frac{1}{4\pi}\e^{4\pi t}
   +\Oh\!\left((1+t)^2\e^{2\pi t}\right),
 \label{eq:beta-large-t}\\
 |\widetilde\alpha_2(\ii t)|+|\widetilde\beta_2(\ii t)|
 &=\Oh\!\left((1+t)^2\e^{2\pi t}\right).
 \label{eq:tilded-large-t}
\end{align}
Thus the hypotheses of \cref{lem:Laplace-transfer} hold on the direct side
with $c=4\pi$ and on the Fourier side with $c=2\pi$.

Away from the interpolation nodes, the common sine-square factor in
\eqref{eq:a2-Laplace}--\eqref{eq:b2-Laplace} cancels, and
\cref{lem:Laplace-transfer}, together with
\eqref{eq:alpha-cusp}--\eqref{eq:beta-cusp}, gives
\begin{equation}\label{eq:R-limit-proof}
 R(r)=\frac{\int_0^\infty\alpha_2(\ii t)\e^{-\pi r^2t}\dd t}
 {\int_0^\infty\beta_2(\ii t)\e^{-\pi r^2t}\dd t}
 \longrightarrow\rho.
\end{equation}
At a direct-side node, \eqref{eq:R-node-value} identifies the smooth extension
with exactly the same ratio of Laplace integrals, so the limit holds without
excluding nodal sequences.  The tilded cusp formulas similarly give
\begin{equation}\label{eq:Rhat-limit-proof}
 \widehat R(r)\longrightarrow\frac{-\rho c_0}{-c_0}=\rho,
\end{equation}
and \eqref{eq:Rhat-node-value} again covers the nodes.

A continuous function with a finite limit at infinity is bounded on a
half-line: it is bounded on a compact initial interval, and the defining limit
bounds it on the complementary tail.  Applying this observation to the smooth
extensions proves the final assertion.
\end{proof}

The second exact cusp coefficient determines from which side the two
quotients approach their common limit.

\begin{proposition}[First exponential correction]\label{prop:quotient-correction}
Set $\kappa:=45\pi^2(713-840\log2)/2^{23/4}>0$.  Then, as
$r\to\infty$,
\begin{align}
 R(r)
 &=\rho+\kappa\e^{-2\pi(\sqrt2-1)r}
   \left(1+\Oh(r^{-1})\right),
 \label{eq:R-refined-asymptotic}\\
 \widehat R(r)
 &=\rho-\kappa\e^{-2\pi(\sqrt2-1)r}
   \left(1+\Oh(r^{-1})\right).
 \label{eq:Rhat-refined-asymptotic}
\end{align}
Consequently $C_*>\rho$.
\end{proposition}

\begin{proof}
Apply the refined part of \cref{lem:Laplace-transfer} to
\eqref{eq:direct-difference-cusp} with $A=\rho c_0$, $B=c_0$, $D=d_1$,
$\eta=2$, and $\zeta=3$.  The ratio representation is
valid at the interpolation nodes as well, by \eqref{eq:R-node-value}.  Hence
\begin{equation}\label{eq:R-minus-rho-J}
 R(r)-\rho=\frac{d_1}{c_0}\frac{J_2(r)}{J_1(r)}(1+o(1)).
\end{equation}
More explicitly, the numerator is $d_1J_2(r)+\Oh(J_3(r))$, while the
denominator is $c_0J_1(r)+\Oh(J_2(r))$.  The resulting relative error is
$\Oh(J_3/J_2)+\Oh(J_2/J_1)$, which is exponentially small by
\eqref{eq:J-ratio-asymptotic}.  The standard expansion
\begin{equation}\label{eq:Bessel-full-first}
 K_\nu(s)=\sqrt{\frac{\pi}{2s}}\e^{-s}
 \left(1+\frac{4\nu^2-1}{8s}+\Oh_\nu(s^{-2})\right)
\end{equation}
therefore gives
\begin{equation}\label{eq:J2-J1-specific}
 \frac{J_2(r)}{J_1(r)}
 =2^{21/4}\e^{-2\pi(\sqrt2-1)r}
  \left(1+\Oh(r^{-1})\right).
\end{equation}
Finally,
$(d_1/c_0)2^{21/4}
 =45\pi^2(713-840\log2)/2^{23/4}=\kappa$,
which proves \eqref{eq:R-refined-asymptotic}.

For the Fourier quotient, \eqref{eq:fourier-difference-cusp} has the same
leading numerator coefficient $d_1$, whereas
\eqref{eq:betatilde-cusp} has leading denominator coefficient $-c_0$.
Using \eqref{eq:Rhat-node-value} at the nodes, the same calculation yields
\eqref{eq:Rhat-refined-asymptotic}.  Since the factor
$1+\Oh(r^{-1})$ is positive for all sufficiently large $r$, one has
$R(r)>\rho$ on a terminal interval.  Therefore $C_*\geq\sup R>\rho$.
\end{proof}

\section{Completion of the conjecture}
\label{sec:completion}

We now combine the analytic construction with Poisson summation.  We use the
standard normalization of the Leech lattice: it is even unimodular, has no
vectors of squared length $2$, and its shortest nonzero vectors have squared
length $4$; there are exactly $196560$ of them \cite{ConwaySloane}.  Thus
$\Leech^*=\Leech$ and $|x|^2\in2\Z_{\geq2}$ for every
$x\in\Leech\setminus\{0\}$.

\begin{proof}[Proof of \cref{thm:main}]
By \cref{prop:quotient-limit}, the threshold $C_*$ is finite.  The exact sign
reduction in \cref{prop:sign-reduction} proves that, for every $C\geq C_*$,
$g_C(r)\leq0$ for $r\geq\sqrt6$ and $\widehat g_C(r)\geq0$ for $r\geq0$.
Equation \eqref{eq:gC-at-2} gives $g_C(2)=1$.

Because $g_C$ is Schwartz and $\Leech$ is self-dual, Poisson summation gives
\begin{equation}\label{eq:Poisson-gC}
 \sum_{x\in\Leech}g_C(x)=
 \sum_{y\in\Leech}\widehat g_C(y).
\end{equation}
On the direct side, \eqref{eq:gC-direct-nodes} annihilates every nonzero
Leech shell except the shortest shell $|x|=2$, on which $g_C(2)=1$.  On the
Fourier side, \eqref{eq:gC-fourier-nodes} annihilates every nonzero Leech
shell.  Hence \eqref{eq:Poisson-gC} reduces exactly to
\begin{equation}\label{eq:Poisson-reduced-196560}
 g_C(0)+196560\,g_C(2)=\widehat g_C(0).
\end{equation}
Since $g_C(2)=1$, this is
\begin{equation}\label{eq:ratio-final}
 \frac{\widehat g_C(0)-g_C(0)}{g_C(2)}=196560.
\end{equation}
Thus all four requirements of the Cohn--Kumar conjecture hold.

The converse assertion within the affine line is exactly the converse part of
\cref{prop:sign-reduction}.  Finally, \cref{prop:quotient-correction} gives
$C_*>\rho$, proving \eqref{eq:Cstar-strict-lower}.
\end{proof}

\begin{proof}[Proof of \cref{cor:theta-shells}]
Fix $m\geq2$.  Poisson summation applies to the Schwartz function $a_m$.
By \eqref{eq:a-data}, its value is $1$ on the Leech shell of squared length
$2m$ and $0$ on every other nonzero Leech shell, while its Fourier transform
vanishes on every nonzero Leech shell.  Hence
$a_m(0)+N_m=\widehat a_m(0)$, which proves \eqref{eq:shell-interpolation}.

For the explicit formula, normalize the theta series by
\begin{equation}\label{eq:theta-normalization}
 \Theta_{\Leech}(z):=\sum_{x\in\Leech}q^{|x|^2/2}
 =1+\sum_{m\geq2}N_mq^m,
 \qquad q=\e^{2\pi\ii z}.
\end{equation}
Because $\Leech$ is even unimodular of rank $24$, its theta series belongs to
$M_{12}(\mathrm{SL}_2(\Z))$.  The standard dimension formula gives
$M_{12}(\mathrm{SL}_2(\Z))=\C E_{12}\oplus\C\Delta$.  The constant term in
\eqref{eq:theta-normalization} forces the coefficient of $E_{12}$ to be $1$.
The Leech lattice has no vectors of squared length $2$, so the coefficient of
$q$ is zero; since $[q]E_{12}=65520/691$ and $[q]\Delta=1$, the coefficient of
$\Delta$ is $-65520/691$.  Thus
\begin{equation}\label{eq:Leech-theta-series}
 \Theta_{\Leech}(z)=E_{12}(z)-\frac{65520}{691}\Delta(z).
\end{equation}
Finally,
\begin{equation}\label{eq:E12-and-Delta-coefficients}
 E_{12}(z)=1+\frac{65520}{691}
 \sum_{n\geq1}\sigma_{11}(n)q^n,
 \qquad
 \Delta(z)=\sum_{n\geq1}\tau(n)q^n.
\end{equation}
Comparing the coefficient of $q^m$ in the last two displays proves
\eqref{eq:shell-ramanujan}.
\end{proof}

\section{Reproducibility and relation to previous work}
\label{sec:scope}

No floating-point computation, numerical optimization, or new
computer-assisted inequality is used in the arguments of this article.
The only elementary scalar estimates required are the exact inequalities
$\pi<22/7$ and $\log 2<7/10$, both proved at the points where they are used.
The global sign properties of the anchor function are taken from
\cite{CKMRV2017}, and the interpolation identities are taken from
\cite{CKMRV2022}.  The new algebraic input consists of the explicit finite
coefficient identities \eqref{eq:Aplus}--\eqref{eq:Bminus} and the
cancellations recorded in
\cref{tab:A-cancellation,tab:B-cancellation}; the subsequent transfer is the
analytic argument of \cref{sec:laplace}.  In particular, no logical step in
the proof depends on numerical or symbolic verification.

Fourier interpolation has developed substantially since the
one-dimensional formula of Radchenko and Viazovska \cite{RV2019}.  Relevant
directions include interpolation and summation formulas in broader settings
\cite{Stoller2020,BRS2023,GV2025}, quantitative and structural questions
concerning exceptional lattices \cite{BRR2024,Cohn2024,Lee2024,Lee2026}, and
recent work on low-dimensional or perturbed Fourier interpolation
\cite{CasseseRamos2025,RadchenkoSun2025,BBRSS2025}.  These works provide
important context and related techniques.  To the best of our knowledge,
however, the affine construction \eqref{eq:gC-definition} and the proof of the
2009 ratio conjecture via bounds for the quotients in
\eqref{eq:R-definitions} do not appear in the previous literature.

\section*{Author and disclosure statements}

\noindent\textbf{Use of computational and generative tools.}
The affine ansatz \eqref{eq:gC-definition}, the quotient strategy in
\cref{sec:laplace}, and portions of the symbolic coefficient extraction in
\cref{sec:kernel,sec:cusp} were developed with assistance from an OpenAI
language model.  Responsibility for the mathematical content and citations
rests with the author.

\end{document}